\documentclass[a4paper]{cas-sc}

\usepackage[numbers]{natbib}
\usepackage{amsmath}
\usepackage{amssymb,amsthm}
\usepackage{mathrsfs}
\usepackage{graphicx}
\usepackage{color}
\usepackage{indentfirst}
\usepackage{float}
\usepackage {comment}
\usepackage{autobreak}
\allowdisplaybreaks[4]

\newtheorem{thm}{Theorem}[section] 
\newtheorem{rem}[thm]{Remark}

\newtheorem{lem}[thm]{Lemma}

\def\tsc#1{\csdef{#1}{\textsc{\lowercase{#1}}\xspace}}
\tsc{WGM}
\tsc{QE}
\tsc{EP}
\tsc{PMS}
\tsc{BEC}
\tsc{DE}

\begin{document}
\let\WriteBookmarks\relax
\def\floatpagepagefraction{1}
\def\textpagefraction{.001}
\shorttitle{}
\shortauthors{Zhanpeng Deng et~al.}

\title [mode = title]{Radon random sampling and  reconstruction in local shift-invariant signal space}                      



\author[1]{Zhanpeng Deng}[type=editor,
                        auid=000,bioid=1]
\ead{dengzhp9@mail2.sysu.edu.cn}
\address[1]{School of Mathematics, Sun Yat-sen University, 510275 Guangzhou, China}

\author[1]{Jiao Li}

\ead{lijiao36@mail2.sysu.edu.cn}


\author%
[1,2]
{Jun Xian}[orcid=https://orcid.org/0000-0003-3772-5622]
\cormark[1]
\ead{xianjun@mail.sysu.edu.cn}

\address[2]{Guangdong Province Key Laboratory of Computational Science, 510275 Guangzhou, China}

\cortext[cor1]{Corresponding author}


\begin{abstract}
In this paper, we deal with the problem of reconstruction from Radon random samples in local shift-invariant signal space. Different from sampling after Radon transform, we consider sampling before Radon transform, where the sample set is randomly selected from a square domain with a general probability distribution. First, we prove that the sampling set is stable with high probability under a sufficiently large sample size. Second, we address the problem of signal reconstruction in two-dimensional computed tomography. We demonstrate that the sample values used for this reconstruction process can be determined completely from its Radon transform data. Consequently, we develop an explicit formula to reconstruct the signal using Radon random samples.
\end{abstract}
\begin{keywords}
Radon transform \sep random sampling \sep shift-invariant signal space \sep reconstruction formula
\end{keywords}
\maketitle
\section{Introduction}
\noindent The Radon transform was proposed by Johann Radon in 1917. For a function $f:\mathbb{R}^2 \to \mathbb{R}$, $\mathbf{x}\in\mathbb{R}^2$ and a direction vector $\mathbf{p}=\left(\cos\theta,\sin\theta\right)$ with $\theta\in\left[0,2\pi\right)$, its Radon transform at $t\in \mathbb{R}$ is obtained by integrating along the line $\mathbf{x}=(x,y)=t\mathbf{p}+s(-\sin\theta,\cos\theta)$,
\begin{equation}
    \mathcal{R}_\mathbf{p}\left(f\left(\mathbf{x}\right)\right)\left(t\right)=\mathcal{R}_\mathbf{p} f\left(t\right)=\int_{-\infty}^{+\infty}f\left(t\cos\theta-s\sin\theta,t\sin\theta+s\cos\theta\right)ds\label{1.1}.
\end{equation}
The main idea of this Radon transform is to define a function $f \left( x, y \right)$ to perform higher-dimensional spatial line integrals along any straight line (or hyperplane) in the plane (or space) \cite{2019Recovering}. As shown in Figure \ref{Fig1}, we can see that the two-dimensional Radon transform of $f$ is actually an integral of $f$ along the line which simulates X-ray passing through objects. \\
\begin{figure}[h]
\centering
\includegraphics[scale=.21]{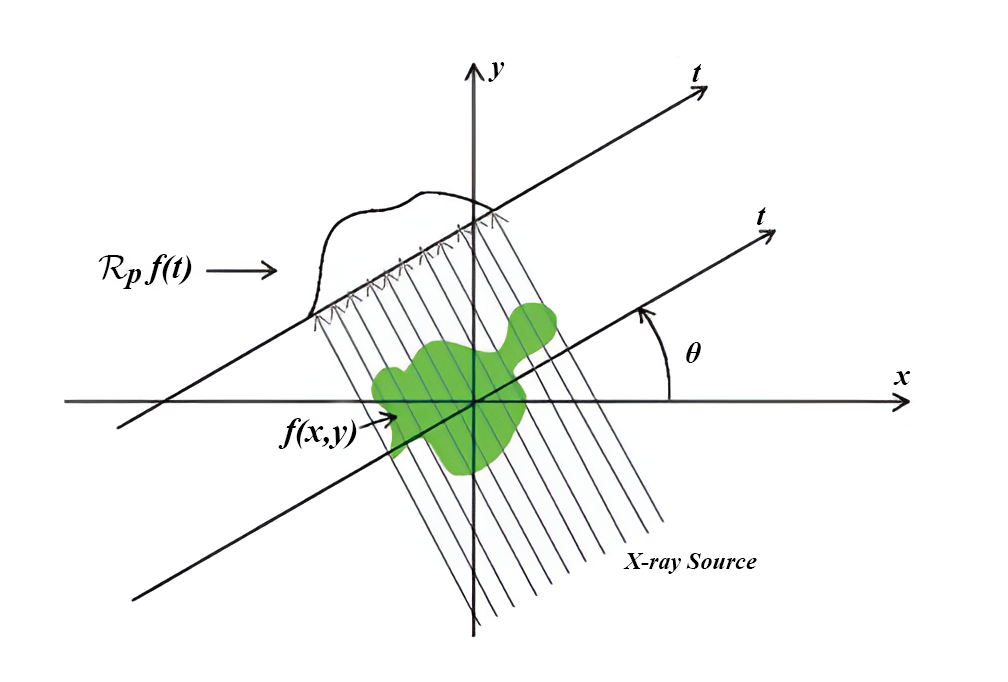}
\caption{Radon transform: $\mathcal{R}_{\mathbf{p}}f\left(t\right)$ is the projection of $f$ along the X-ray at angle $\theta$.}\label{Fig1}
\end{figure}
\indent The Radon transform plays a fundamental role in computed tomography (\textbf{CT}) imaging \cite{appl2017,NSTLAD211C0BE7E08794695A2123086A9B23}. During CT scans, X-rays are used to acquire multi-angle projection data of human tissues, which mathematically correspond to Radon transform projections. The fundamental problem of CT is reconstructing the
 function $f$ by its Radon transform projections \cite{XU2006388}. In \cite{WOS:000178752800002}, the authors developed a Fourier-based algorithm for non-standard sampling in Radon transform reconstruction. The authors in \cite{2023Random1} studied reconstructing measurable functions in locally compact Abelian groups using random measures. In addition to CT imaging, the Radon transform is also applied in fields such as seismology, astronomy and medical diagnostics \cite{2025A, 2024A,1983The}. \\
\indent The sampling problem aims to recover a function $f$ from the sampled values on some sampling set \cite{2010A}. To deal with this problem, we need to specify the signal space. The shift-invariant space, capable of representing spectrally smooth signals and ensuring numerical feasibility, provides a robust framework for modeling biomedical images in \textbf{CT}. The continuous-domain representation of biomedical images can be expressed as functions in this space, thereby enabling the tackling of image reconstruction challenges in \textbf{CT} \cite{Entezari20121532}. While most existing research focuses on global sampling within this space, reconstruction from local samples has often been seen as a highly effective method in numerous signal processing tasks \cite{appl2009}. Let $\mathcal{E}_K=\left[-K,K\right]^2$ and $E=\left[\lceil -N-K\rceil,\lfloor K+N\rfloor\right]^{2}\cap \mathbb{Z}^2$, where $N$ and $K$ are positive numbers. Suppose that $\mathcal{Q}$ is the cardinality of $E$. We denote $E$ by $\{\mathbf{k}_{1},\cdots,\mathbf{k}_{\mathcal{Q}}\}$ and consider the problem of reconstruction from Radon random samples in local shift-invariant signal space
\begin{equation}
\mathcal{S} _{N,K}\left(\varphi\right)=\left\{f:f\left(\mathbf
{x}\right)= \sum_{l=1}^{\mathcal{Q}}{c_{\mathbf{k}_l}}\varphi \left( \mathbf
{x} -\mathbf{k}_l \right),\;\{ c_{\mathbf{k}_l} \} \in \ell ^2\left(\left[ 1,\mathcal{Q} \right]\right),\;\mathbf{x}\in \mathcal{E}_K \right\},\label{1.2}
\end{equation}
 where the generator $\varphi$ with supp$\left(\varphi\right)\subseteq \left[-N,N\right]^2$ is a continuous function with stable shifts. \\
\indent For the above signal space, the key problem lies in determining a sampling strategy that ensures a stable reconstruction of the signal $f$. Typically, uniform and non-uniform sampling are the primary consideration \cite{2023Determination,2013A,2019A}. However, compared to the above methods, random sampling has greater representativeness and operational simplicity. Due to these advantages, random sampling has become a flexible and widely used method. Random sampling has also been extensively applied in compressed sensing, image processing and learning theory \cite{2014Monte,2001On} in recent years.
Extensive research has been conducted on the sampling and reconstruction of various 
 random signals \cite{2019Reconstruction,2023Random,2023AA,2013AA,2021A}.\\
\indent In this paper, we restrict the domain of $ f\in \mathcal{S} _{N,K}\left(\varphi\right)$ on the interval $[-K,K]^{2}$, and then we actually deal with the problem of reconstruction based on Radon random samples in local shift-invariant signal space. Specifically, the sampling set $X=\left\{\mathbf{x}_{j}\right\}_{j=1}^{n}\subseteq \mathcal{E}_K=\left[-K,K\right]^2$ is randomly chosen, where $K$ is a positive number. We consider that $\mathcal{R}_{\mathbf{p}}\left(f\left(\mathbf{x}\right)\right)\left(\mathbf{px}_j\right)$ is the integral of the function $f$ along a line with direction vector $\mathbf{p}=(\cos\theta,\sin\theta)$ passing through $\mathbf{x}_j = \left(x_j,y_j\right) \in X$, $j=1,\ldots,n$, where $\theta\in[0,2\pi)$. Therefore, a straight line in the image space is transformed into a salient pixel in the sinogram \cite{1981Hough}.  By (\ref{1.1}) and $\mathbf{x}_j=\left(x_j,y_j\right)=\left(t_j\cos\theta-s_j\sin\theta,t_j\sin\theta+s_j\cos\theta\right)$, we obtain 
\begin{equation}
    \nonumber
    \begin{pmatrix}
    t_j\\s_j
    \end{pmatrix}
=\begin{pmatrix}
        \cos\theta & \sin\theta\\
        -\sin\theta & \cos\theta\\
\end{pmatrix}
\cdot
\begin{pmatrix}
    x_j\\y_j
\end{pmatrix},
\end{equation}
where $t_j=\mathbf{px}_j$, $j=1,\ldots,n$. We denote $ \mathcal{R}_{\mathbf{p}}f\left(t\right):=\mathcal{R}_{\mathbf{p}}\left(f\left(\mathbf{x}\right)\right)\left(t\right)$ and $\mathcal{R}_{\mathbf{p}}\left(f\left(\mathbf{x}_j\right)\right):=\mathcal{R}_{\mathbf{p}}\left(f\left(\mathbf{x}\right)\right)\left(\mathbf{px}_j\right).$ Then
 we have 
 \begin{equation}
 \mathcal{R}_{\mathbf{p}}f\left(t_j\right)=\mathcal{R}_{\mathbf{p}}\left(f\left(\mathbf{x}\right)\right)\left(t_j\right)=\mathcal{R}_{\mathbf{p}}\left(f\left(\mathbf{x}\right)\right)\left(\mathbf{px}_j\right)=\mathcal{R}_{\mathbf{p}}\left(f\left(\mathbf{x}_{j}\right)\right), \label{new1.4}  
 \end{equation} The available sampling values are in the form of $\left\{ \mathcal{R}_{\mathbf{p}}\left(f\left(\mathbf{x}_{j}\right)\right),\mathbf{x}_{j}\in X \right\}$. The stability of the sampling set is critical for reconstructing $f$, as only Radon samples obtained from a stable set ensure reliable signal recovery \cite{2010A}. For any $f\in\mathcal{S}_{N,K}\left(\varphi\right)$, the stable sampling set $X=\left\{\mathbf{x}_j\right\}_{j=1}^{n}$ is of the form
 \begin{equation}
     c\left\|f\right\|_{L^2\left(\mathcal{E}_K\right)}^2\leqslant\sum_{j=1}^n\left|\mathcal{R}_\mathbf{p}\left(f\left(\mathbf{x}_j\right)\right)\right|^2\leqslant C\left\|f\right\|_{L^2\left(\mathcal{E}_K\right)}^2,\label{1.4}
 \end{equation}
 where $c$ and $C$ are positive constants. As a result of the random sampling process, there is a certain probability that the random sampling set is stable. Then the signal $f\in\mathcal{S}_{N,K}\left(\varphi\right)$ can be recovered via our reconstruction formula. 
 \\
\indent  This paper is organized as follows. In section 2, we first introduce some foundational content for some assumptions, then we establish the sufficient and necessary condition that $f\in \mathcal{S} _{N,K}\left(\varphi\right)$ can
be entirely determined by the sampling set $\left\{\mathcal{R}_{\mathbf{p}}\left(f\left(\mathbf{x}_j\right)\right)\right\}_{j=1}^{n}$. In section 3, we consider the matrix Bernstein inequality which will help us derive the Radon random sampling inequality. In section 4, we establish the primary outcome of reconstruction based on Radon random samples.  In section 5, we perform some numerical tests to verify the effectiveness of the reconstruction formula. In section 6, we conclude the whole paper.    

\section{Preliminary}
\numberwithin{equation}{section}
\noindent In this section, we propose a necessary and sufficient condition under which all functions $f\in \mathcal{S}_{N,K}\left(\varphi\right)$ can be determined completely by their Radon samples at $X=\left\{\mathbf{x}_{j}\right\}_{j=1}^{n}\subseteq \mathcal{E}_K$ in Theorem \ref{thm2.1}. Firstly, we introduce some preliminary knowledge.\\
\indent Throughout the paper, the generator $\varphi$ is continuous and has compact support contained in $\left[-N,N \right]^2$. The signal domain is defined as $\mathbf{x}\in \mathcal{E}_K=\left[-K,K\right]^2$. Then for any signal $f$ in the shift-invariant space generated by $\varphi$, there exists a finite sequence $\{c_{\mathbf{k}_{l}}\in\mathbb{R},\;l=1,\ldots,\mathcal{Q}\}$ such that $f$ can be written as follows
\begin{equation}
    f\left(\mathbf{x}\right)=\sum_{l=1}^{\mathcal{Q}}{c_{\mathbf{k}_{l}}\varphi\left(\mathbf{x}-\mathbf{k}_{l}\right)}, \;\forall\;\mathbf{x}\in \mathcal{E}_K,\label{2.1}
\end{equation}
where 
\begin{equation}
    E:=\left\{\mathbf{k}_{1},\mathbf{k}_{2},\ldots,\mathbf{k}_{\mathcal{Q}}\right\}=\left[\lceil -N-K\rceil,\lfloor N+K\rfloor\right]^{2}\cap \mathbb{Z}^2\label{2.2},
\end{equation} $\mathcal{Q}$ is the cardinality of $E$ and $N$, $K$ are positive numbers. In the whole paper, we assume that $\mathcal{Q}>1$.\\
\indent Our assumptions regarding the generator and the probability density function, as well as their corresponding constants, are outlined below:
\begin{enumerate}
\item [\textbf{(A.1)}] 
The generator $\varphi$ is a continuous function with compact support $\left[-N,N\right]^2$ and has stable shifts, i.e.
\begin{equation}
\nonumber m_{2}\left\|C\right\|_{\ell^2}\leqslant \left\|\sum_{l=1}^{\mathcal{Q}}c_{\mathbf{k}_{l}}\varphi\left(\cdot-\mathbf{k}_{l}\right)\right\|_{L^2\left(\mathcal{E}_K\right)}\leqslant M_{2}\left\|C\right\|_{\ell^2},
\end{equation}
where $0< m_{2}\leqslant M_{2}<\infty$. 
\item [\textbf{(A.2)}] 
Suppose that $\xi$ is a probability density function over $\mathcal{E}_K$ and satisfies the following condition   
\begin{equation}
\nonumber 0< C_{\xi,l}\leqslant\xi\left(\mathbf{x}\right)\leqslant C_{\xi,u} ,\;\forall\; \mathbf{x}\in \mathcal{E}_K.
\end{equation}
\end{enumerate}



        

 Let $X=\left\{\mathbf{x}_{1},\ldots,\mathbf{x}_{n}\right\}\subseteq \mathcal{E}_K$ be the sampling set. To address noise-induced degradation in $f(\mathbf{x}_j)$, we reconstruct 
$f$ via its high-accuracy Radon samples on $X$, and establish a necessary and sufficient condition for exact recovery of any $f\in \mathcal{S} _{N,K}\left(\varphi\right)$ in (\ref{1.2}) from these samples.
\begin{thm}
    Suppose that $\varphi\in L^2\left(\mathbb{R}^2\right)$ satisfying $(\mathbf{A.1})$ and $\left\{\varphi\left(\cdot-\mathbf{k}\right):\mathbf{k}\in\mathbb{Z}^2 \right\}$ is linearly independent. Let the direction vector be $\mathbf{p}=\left(\cos\theta,\sin\theta\right)$ 
such that $\mathcal{R}_\mathbf{p}\left(\varphi\right)$ is continuous. Let $E=\left[\lceil -N-K\rceil,\lfloor K+N\rfloor\right]^{2}\cap \mathbb{Z}^2:=\left\{\mathbf{k}_{1},\ldots,\mathbf{k}_{\mathcal{Q}}\right\}$, $\mathcal{Q}$ be the cardinality of $E$ and\label{thm2.1} 
    \begin{equation}
         \nonumber U_{\varphi,\mathbf{p},X}:=
         \begin{pmatrix}
        \mathcal{R}_\mathbf{p}\left(\varphi\left(\mathbf{x}_{1}-\mathbf{k}_{1}\right)\right) & \mathcal{R}_\mathbf{p} \left(\varphi\left(\mathbf{x}_{1}-\mathbf{k}_{2}\right)\right)&\cdots&\mathcal{R}_\mathbf{p} \left(\varphi\left(\mathbf{x}_{1}-\mathbf{k}_{\mathcal{Q}}\right)\right)\\
         \mathcal{R}_\mathbf{p} \left(\varphi\left(\mathbf{x}_{2}-\mathbf{k}_{1}\right)\right) & \mathcal{R}_\mathbf{p} \left(\varphi\left(\mathbf{x}_{2}-\mathbf{k}_{2}\right)\right)&\cdots&\mathcal{R}_\mathbf{p} \left(\varphi\left(\mathbf{x}_{2}-\mathbf{k}_{\mathcal{Q}}\right)\right)\\
         \vdots &\vdots &\ddots&\vdots\\
          \mathcal{R}_\mathbf{p} \left(\varphi\left(\mathbf{x}_{n}-\mathbf{k}_{1}\right)\right) & \mathcal{R}_\mathbf{p} \left(\varphi\left(\mathbf{x}_{n}-\mathbf{k}_{2}\right)\right)&\cdots&\mathcal{R}_\mathbf{p} \left(\varphi\left(\mathbf{x}_{n}-\mathbf{k}_{\mathcal{Q}}\right)\right)\\
        \end{pmatrix}.
    \end{equation}
    Then for sampling set $X=\left\{\mathbf{x}_{j}\right\}_{j=1}^{n}\subseteq \mathcal{E}_{K}$, $f\in \mathcal{S} _{N,K}\left(\varphi\right)$ can be determined completely by its 
 Radon $\left(w.r.t.\, \mathbf{p}\right)$ samples if and only if the $\mathcal{Q}\times\mathcal{Q}$ matrix $U_{\varphi,\mathbf{p},X}^\mathrm{T} U_{\varphi,\mathbf{p},X}$ is invertible.
\end{thm}
\begin{proof}
    $\left(\Leftarrow \right)$ We notice that $U_{\varphi,\mathbf{p},X}^\mathrm{T} U_{\varphi,\mathbf{p},X}$ is invertible, then we can prove that $\left\{\mathcal{R}_\mathbf{p}\left(\varphi\left(\cdot-\mathbf{k}_{l}\right)\right):l=1,\ldots,\mathcal{Q}\right\}$ is linearly independent in $L^2\left(\mathbb{R}\right)$.\\ \indent In fact, assume that it is not linearly independent, there is a nonzero sequence $\left\{\widetilde{c}_{l}\right\}_{l=1}^{\mathcal{Q}}\in \ell^{2} $  satisfying
    \begin{equation}
        \nonumber\left\|\sum_{l=1}^{\mathcal{Q}} \widetilde{c}_{l}\mathcal{R}_\mathbf{p}\left(\varphi\left(\cdot-\mathbf{k}_{l}\right)\right)\right\|_{L^{2}\left(\mathbb{R}\right)}^2=\int_{\mathbb{R}}\left|\sum_{l=1}^{\mathcal{Q}}\widetilde{c}_{l}\mathcal{R}_\mathbf{p}\left(\varphi\left(\mathbf{x}-\mathbf{k}_{l}\right)\right)\left(t\right)\right|^2 dt=0.
    \end{equation}
    Due to $\mathcal{R}_\mathbf{p}\left(\varphi\right)$ is continuous and $t=\mathbf{px}$ in $ \left(\ref{new1.4}\right)$, for any $j\in\left\{1,\ldots,n\right\}$ and $\mathbf{x}_{j}\in X$, we obtain $\sum_{l=1}^{\mathcal{Q}}\widetilde{c}_{l}\mathcal{R}_\mathbf{p}\left(\varphi\left(\mathbf{x}_j-\mathbf{k}_{l}\right)\right)=0$ which implies that the matrix $U_{\varphi,\mathbf{p},X}^\mathrm{T} U_{\varphi,\mathbf{p},X}$ is not invertible. This contradicts with the assumption. 
    \par From (\ref{2.1}), there exists
    $\{c_{\mathbf{k}_{l}},l=1,\ldots,\mathcal{Q}\}\in\mathbb{R}$ such that the equation 
    \begin{equation}
        \mathcal{R}_\mathbf{p}\left(f\left(\mathbf{x}\right)\right)(t)=\sum_{l=1}^{\mathcal{Q}}c_{\mathbf{k}_{l}}\mathcal{R}_\mathbf{p}\left(\varphi\left(\mathbf{x}-\mathbf{k}_{l}\right)\right)(t),\;\forall\; t\in\left[-\sqrt{2}K,\sqrt{2}K\right]\label{2.5}
    \end{equation}
    is true. Next, we solve the finite linear system for the coefficients $\{c_{\mathbf{k}_l}\}_{l=1}^{\mathcal{Q}}$,
    \begin{equation}
        \nonumber\sum_{l=1}^{\mathcal{Q}}c_{\mathbf{k}_l}\mathcal{R}_{\mathbf{p}}\left(\varphi\left(\mathbf{x}_j-\mathbf{k}_{l}\right)\right)=\mathcal{R}_{\mathbf{p}}\left(f\left(\mathbf{x}_j\right)\right),\,j=1,\ldots,n,\,\mathbf{x}_j\in \mathcal{E}_K.
    \end{equation}
    Therefore, we have the matrix form
    \begin{equation}
    \nonumber U_{\varphi,\mathbf{p},X}(c_{\mathbf{k}_{1}},\ldots,c_{\mathbf{k}_{\mathcal{Q}}})^\mathrm{T}=\mathcal{R}_\mathbf{p}\left(f\left(\mathbf{x}_{j}\right)\right)_{j=1,\ldots,n},
    \end{equation}
    then if the matrix $U_{\varphi,\mathbf{p},X}^\mathrm{T} U_{\varphi,\mathbf{p},X}$ is invertible, the coefficients $(c_{\mathbf{k}_{1}},\ldots,c_{\mathbf{k}_{\mathcal{Q}}})^\mathrm{T}$ can be entirely reconstructed from its Radon samples $\left\{\mathcal{R}_\mathbf{p}\left(f\left(\mathbf{x}_{j}\right)\right),j=1,\ldots,n\right\}$ in the following formula
    \begin{equation}
    \nonumber(c_{\mathbf{k}_{1}},\ldots,c_{\mathbf{k}_{\mathcal{Q}}})^\mathrm{T}=\left(U_{\varphi,\mathbf{p},X}^\mathrm{T} U_{\varphi,\mathbf{p},X}\right)^{-1}U_{\varphi,\mathbf{p},X}^{\mathrm{T}}\left(\mathcal{R}_\mathbf{p}\left(f\left(\mathbf{x}_{j}\right)\right)\right)_{j=1,\ldots,n}.
    \end{equation}
    \indent Thus, due to the fact that $\left\{\varphi\left(\cdot-\mathbf{k}\right),\mathbf{k}\in E\right\}$ is linearly independent, $f$ can be determined uniquely by its Radon $\left(w.r.t.\, \mathbf{p}\right)$ samples at $X=\left\{\mathbf{x}_{1},\ldots,\mathbf{x}_{n}\right\}$.\\
    $\left(\Rightarrow \right)$ If matrix $U_{\varphi,\mathbf{p},X}^{\mathrm{T}} U_{\varphi,\mathbf{p},X}$ is not invertible, the coefficients $(c_{\mathbf{k}_{1}},\ldots,c_{\mathbf{k}_{\mathcal{Q}}})^\mathrm{T}$ can not be determined completely. Since $\left\{\varphi\left(\cdot-\mathbf{k}\right),\mathbf{k}\in E\right\}$ is linearly independent, then by $ f=\sum_{l=1}^{\mathcal{Q}}{c_{\mathbf{k}_{l}}\varphi\left(\cdot-\mathbf{k}_{l}\right)}$ in (\ref{2.1}), we know that this contradicts with the fact that $f$ can be determined uniquely.
\end{proof}

\section{Random sampling inequalities for the Radon transform in $\mathcal{S} _{N,K}(\varphi)$}
\noindent  In this section, we consider the random sampling inequalities for the Radon transform in $\mathcal{S} _{N,K}(\varphi)$. First, we need to explain the advantages of the sampling method. \\
\indent The defect classification problem in image processing encompasses multiple stages: image acquisition, pre-processing, segmentation and surface defect identification. By applying our sampling strategy, targeted analysis of specific angles and positions becomes feasible. For example, focusing on critical sampling points (e.g., on aircraft wings) and selected directions allows efficient computation of their corresponding Radon transform projections. Unlike full-image approaches that demand substantial computational resources, our method preserves reconstruction accuracy while significantly reducing computational and storage requirements.\\
\indent We observe that if $f\in \mathcal{S} _{N,K}\left(\varphi\right)$ in (\ref{1.2}) satisfies the stability condition (\ref{1.4}), the normalized function $f/\left\|f\right\|_{L^2\left(\mathcal{E}_K\right)}$ will also satisfy (\ref{1.4}).
Therefore, we define the normalized space as
\begin{equation}
    \widetilde{\mathcal{S}}_{N,K}\left(\varphi\right)=\left\{f\in \mathcal{S} _{N,K}\left(\varphi\right):\left\|f\right\|_{L^2\left(\mathcal{E}_K\right)}=1\right\}\label{3.1}.
\end{equation}
\indent Finally, we establish a series of inequalities essential for Theorem \ref{thm3.4}. 
\begin{lem}
    Let $\widetilde{\mathcal{S}}_{N,K}\left(\varphi\right)$ and $m_{2}$ be defined by (\ref{3.1}) and $(\mathbf{A.1})$ respectively\label{lem3.1}. Then, for any $f\in \widetilde{\mathcal{S}}_{N,K}\left(\varphi\right)$, we have 
    \begin{align}
        \left\|\mathcal{R}_{\mathbf{p}}\left(f\right)\right\|_{L^{\infty}{\left(\mathcal{E}_K\right)}}&\leqslant2KC_{\theta}\frac{C_{\varphi}}{m_{2}},\label{3.2}\\
        \left\|\mathcal{R}_{\mathbf{p}}\left(f\right)\right\|_{L^2\left(\mathcal{E}_K\right)}&\leqslant2KC_{\theta},\label{3.3}
    \end{align}
    where $C_{\varphi}=sup_{\mathbf{x}\in \mathcal{E}_K}\sum_{l=1}^{\mathcal{Q}}{\left|\varphi\left(\mathbf{x}-\mathbf{k}_l\right)\right|}$ and $C_{\theta}=\sin\theta+\cos\theta$ with $\theta\in \left[0,\pi/2\right)$.
\end{lem}
\begin{proof}
Let $\mathbf{x}=\left(x,y\right)\in \mathcal{E}_K$, $t=x\cos\theta+y\sin\theta$ and $s=-x \sin\theta+y\cos\theta$. Then, we can calculate the value range of $t$ and $s$. For $\theta\in\left[0,\pi/2\right)$, we obtain
    $x=t\cos\theta-s\sin\theta$ and $y=t\sin\theta+s\cos\theta$. It follows from $\mathbf{x}=\left(x,y\right)\in \mathcal{E}_K$ and $\sin\theta,\,\cos\theta>0$ that
\begin{equation}
    \nonumber t,s\in\left[-K\left(\sin\theta+\cos\theta\right),K\left(\sin\theta+\cos\theta\right)\right].
\end{equation}
    Similarly, we can solve the cases when $\theta$ is in the other three quadrants. We summarize it as follows
    \[
    t,s\in
    \begin{cases}
\left[-K\left(\sin\theta+\cos\theta\right),K\left(\sin\theta+\cos\theta\right)\right],&\theta\in\left[0,\pi/2\right),\\
        \left[-K\left(\sin\theta-\cos\theta\right),K\left(\sin\theta-\cos\theta\right)\right],&\theta\in\left[\pi/2,\pi\right),\\
        \left[-K\left(-\sin\theta-\cos\theta\right),-K\left(\sin\theta+\cos\theta\right)\right],&\theta\in\left[\pi,3\pi/2\right),\\
        \left[-K\left(\cos\theta-\sin\theta\right),K\left(\cos\theta-\sin\theta\right)\right],&\theta\in\left[3\pi/2,2\pi\right).\\
    \end{cases}
    \]
    For the above four cases, the images of the corresponding function are the same on different domains, so we only consider $\theta\in\left[0,\pi/2\right)$.
    For any $f\in \widetilde{\mathcal{S}}_{N,K}\left(\varphi\right)$ in (\ref{3.1}), there exists $C=(c_{\mathbf{k}_1},\ldots,c_{\mathbf{k}_{\mathcal{Q}}})^{\mathrm{T}}$ such that $$f\left(\mathbf{x}\right)=\sum_{l=1}^{\mathcal{Q}}c_{\mathbf{k}_{l}}\varphi\left(\mathbf{x}-\mathbf{k}_{l}\right),\; \forall\;\mathbf{x}\in \mathcal{E}_K.$$ 
    Consequently, for $C_{\theta}=\sin\theta+\cos\theta$ and $ \theta\in \left[0,\pi/2\right)$, the following inequality holds
    \begin{align}
        \nonumber\left\|\mathcal{R}_{\mathbf{p}}\left(f\right)\right\|_{L^{\infty}{\left(\mathcal{E}_K\right)}}&=\underset{t\in[-KC_{\theta},KC_{\theta}]}{sup}\left|\int_{\mathbb{R}}f\left(t\cos\theta-s\sin\theta,t\sin\theta+s\cos\theta\right) ds \right|\\
        \nonumber&\leqslant2KC_{\theta}\left\|f\right\|_{L^{\infty}\left(\mathcal{E}_K\right)}\\
        \nonumber&\leqslant 2KC_{\theta}\underset{\mathbf{x}\in \mathcal{E}_K}{sup}\sum_{l=1}^{\mathcal{Q}}\left|c_{\mathbf{k}_{l}}\varphi\left(\mathbf{x}-\mathbf{k}_{l}\right)\right|\\
        \nonumber&\leqslant 2KC_{\theta}\underset{\mathbf{x}\in \mathcal{E}_K}{sup}\left(\sum_{l=1}^{\mathcal{Q}}\left|c_{\mathbf{k}_l}\right|^2\right)^{\frac{1}{2}}\left(\sum_{l=1}^{\mathcal{Q}}\left|\varphi\left(\mathbf{x}-\mathbf{k}_{l}\right)\right|^2\right)^{\frac{1}{2}}\\
        \nonumber&\leqslant 2KC_{\theta}\frac{\left\|f\right\|_{L^2\left(\mathcal{E}_K\right)}}{m_2}\underset{\mathbf{x}\in \mathcal{E}_K}{sup}\sum_{l=1}^{\mathcal{Q}}\left|\varphi\left(\mathbf{x}-\mathbf{k}_{l}\right)\right|\\
        \nonumber&\leqslant 2KC_{\theta}\frac{C_{\varphi}}{m_2},
    \end{align}
where $f\in\widetilde{\mathcal{S}}_{N,K}\left(\varphi\right)$ in (\ref{3.1}), then $\left\|f\right\|_{L^2\left(\mathcal{E}_K\right)}=1$ and $C_{\varphi}=sup_{\mathbf{x}\in \mathcal{E}_K}\sum_{l=1}^{\mathcal{Q}}{\left|\varphi\left(\mathbf{x}-\mathbf{k}_l\right)\right|}$. Therefore, one has $$\left\|\mathcal{R}_{\mathbf{p}}\left(f\right)\right\|_{L^{\infty}{\left(\mathcal{E}_K\right)}}\leqslant2KC_{\theta}\left\|f\right\|_{L^{\infty}\left(\mathcal{E}_K\right)}\leqslant2KC_{\theta}\frac{C_{\varphi}}{m_{2}}.$$
    \par Next, we estimate $\left\|\mathcal{R}_{\mathbf{p}}\left(f\right)\right\|_{L^2\left(\mathcal{E}_K\right)}$. By $\mathcal{R}_{\mathbf{p}}\left(f\left(\mathbf{x}\right)\right)=\mathcal{R}_\mathbf{p}f\left(t\right)$, we obtain
    \begin{align}
        \left\|\mathcal{R}_{\mathbf{p}}\left(f\right)\right\|_{L^2\left(\mathcal{E}_K\right)}^2\nonumber &=\int_{\mathcal{E}_K}\left|\mathcal{R}_{\mathbf{p}}\left(f\left(\mathbf{x}\right)\right)\right|^2d\mathbf{x}\\
    
        \nonumber&\leqslant2KC_{\theta}\int_{-KC_{\theta}}^{KC_{\theta}}\int_{-KC_{\theta}}^{KC_{\theta}}\left(\int_{\mathbb{R}}\left|f\left(t\cos\theta-s\sin\theta,t\sin\theta+s\cos\theta\right)\right|^2ds\right)dsdt\\
        \nonumber&=2KC_{\theta}\int_{-KC_{\theta}}^{KC_{\theta}}\left[\int_{-KC_{\theta}}^{KC_{\theta}}\int_{\mathbb{R}}\left|f\left(t\cos\theta-s\sin\theta,t\sin\theta+s\cos\theta\right)\right|^2dsdt\right]ds\\
        \nonumber&\leqslant4K^2C_{\theta}^2.
    \end{align}
\end{proof}

\indent In what follows, we derive the probability inequality for the function $f\in \widetilde{\mathcal{S}}_{N,K}\left(\varphi\right)$ using matrix Bernstein inequality in Lemma \ref{lem3.3}. Subsequently, we demonstrate the sampling inequality for $f\in \mathcal{S} _{N,K}\left(\varphi\right)$ in Theorem \ref{thm3.4}.\\
\indent Let $X=\left\{\mathbf{x}_{j},\;j\in \mathbb{N}\right\}$ be a set of independent random variables following a general probability distribution over $\mathcal{E}_K=[-K,K]^2$ with density function $\xi$ satisfies assumption $\textbf{(A.2)}$. Then for any $f\in \widetilde{\mathcal{S}}_{N,K}\left(\varphi\right)$, we define 
\begin{equation} Y_{j}\left(\mathcal{R}_\mathbf{p}\left(f\right)\right)=\left|\mathcal{R}_\mathbf{p}\left(f\left(\mathbf{x}_{j}\right)\right)\right|^2-\int_{\mathcal{E}_K}\xi\left(\mathbf{x}\right) \left|\mathcal{R}_\mathbf{p}\left(f\left(\mathbf{x}\right)\right)\right|^2 d\mathbf{x}.\label{3.4}
\end{equation}
By the above definition, we can see that $\left\{Y_j\left(\mathcal{R}_\mathbf{p}\left(f\right)\right),\;j\in \mathbb{N}\right\}$ is a sequence of independent random variables and its expectation satisfies $\mathbb{E} \left( Y_j\left(\mathcal{R}_{\mathbf{p}}\left( f \right)\right) \right) =0$. 
The matrix Bernstein inequality is crucial in probability theory and statistics, enabling us to derive probabilistic bounds on the norm of the sum of random matrices.
\begin{lem}[Matrix Bernstein inequality \cite{WOS:000306433100001}] Let $X_1,\ldots,X_n$ represent a sequence of independent random self-adjoint matrices of dimension $\mathcal{Q}$. Suppose that each random matrix
satisfies
\begin{equation*}
\mathbb{E}(X_j) = 0\quad and\quad \left\|X_j\right\|\leqslant B,\quad j=1,\ldots,n.
\end{equation*}
Then for all $u>0$, $$\mathbb{P}\left(\lambda_{max}\left(\sum_{j=1}^n X_j\right)\geqslant u\right)\leqslant \mathcal{Q}exp\left(-\frac{u^2 /2}{\sigma^2+Bu/3}\right)$$ holds, where $\lambda_{max}\left(U\right)$ represents the largest singular value of a matrix $U$, $\left\|U\right\|=\Big(\lambda_{max}\left(U^{\mathrm{T}}U\right)\Big)^{\frac{1}{2}}$ denotes the operator norm, and $\sigma^2=\left\|\sum_{j=1}^{n}\mathbb{E}\left(X_j^2\right)\right\|$.\label{lem3.2}
\end{lem}
The random matrices being studied are generated as follows: For each $j\in\mathbb{N}$ and ${l_m,l_n\in \left\{1,\ldots,\mathcal{Q}\right\}}$, we define the $\mathcal{Q}\times\mathcal{Q}$ random matrix $\varPsi_j$,\begin{equation}
    \left(\varPsi_j\right)_{l_m,l_n}=\mathcal{R}_{\mathbf{p}}\left(\varphi\left(\mathbf{x}_j-\mathbf{k}_{l_m}\right)\right)\overline{\mathcal{R}_{\mathbf{p}}\left(\varphi\left(\mathbf{x}_j-\mathbf{k}_{l_n}\right)\right)}, \label{3.5}
\end{equation}
where $\{\mathbf{x}_j\}$ are independent and identically distributed random variable chosen from $\mathcal{E}_K$. Let
 \begin{equation}
X_j=\varPsi_j-\mathbb{E}\left(\varPsi_j\right).\label{3.6}
\end{equation}
\indent Using Lemma \ref{lem3.2}, we derive the probability inequality for all functions in 
$\widetilde{\mathcal{S}}_{N,K}\left(\varphi\right)$.
\begin{lem}
    Let $\widetilde{\mathcal{S}}_{N,K}\left(\varphi\right)$ be given by (\ref{3.1}) and $X=\left\{\mathbf{x}_j\right\}_{j=1}^n$ be a set of independent random variables derived from a general probability distribution on $\mathcal{E}_K$. The density function $\xi$, along with the constants $C_{\xi,u}$ and $C_{\xi,l}$, satisfies assumption $\mathbf{(A.2)}$. For some $n\in\mathbb{N}$ and $\lambda\geqslant0$, the probability
    \begin{align}
        &\nonumber\mathbb{P}\left(\underset{f\in \widetilde{\mathcal{S}}_{N,K}\left(\varphi\right)}{sup}\left|
        \sum_{j=1}^{n}Y_j\left(\mathcal{R}_\mathbf{p}\left(f\right)\right)\right|\geqslant\lambda\right)\leqslant \mathcal{Q}exp\left(-\frac{m_2^4\lambda^2 }{8K^2C_{\theta}^2M_2^2\left[4K^2C_{\theta}^2C_{\varphi}^2C_{\xi,u}n+\frac{\lambda \left(C_{\varphi}^2+C_{\xi,u}m_2^2\right)}{3}\right]}\right)
    \end{align}
   holds, where $m_2,\,M_2$ are the constants defined in $\mathbf{(A.1)}$ and $C_{\varphi},\;C_{\theta}$ are defined in Lemma \ref{lem3.1}.\label{lem3.3}
\end{lem}
\begin{proof}
    By the definition of $\varPsi_j$ in (\ref{3.5}), then we derive that
    \begin{align}
    \nonumber\left(\mathbb{E}\left(\varPsi_j\right)\right)_{l_m,l_n}&=\int_{\mathcal{E}_K}\xi\left(\mathbf{x}\right)\mathcal{R}_{\mathbf{p}}(\varphi(\mathbf{x}-\mathbf{k}_{l_m}))\overline{\mathcal{R}_{\mathbf{p}}(\varphi(\mathbf{x}-\mathbf{k}_{l_n}))}d\mathbf{x}\\
    \nonumber&=\int_{\mathcal{E}_K}\sqrt{\xi(\mathbf{x})}\mathcal{R}_{\mathbf{p}}(\varphi(\mathbf{x}-\mathbf{k}_{l_m}))\overline{\sqrt{\xi(\mathbf{x})}\mathcal{R}_{\mathbf{p}}(\varphi(\mathbf{x}-\mathbf{k}_{l_n}))}d\mathbf{x}\\
    \nonumber&=\left<\sqrt{\xi(\mathbf{x})}\mathcal{R}_{\mathbf{p}}(\varphi(\mathbf{x}-\mathbf{k}_{l_m})),\sqrt{\xi(\mathbf{x})}\mathcal{R}_{\mathbf{p}}(\varphi(\mathbf{x}-\mathbf{k}_{l_n}))\right>.
    \end{align}
   Let $C=(c_{\mathbf{k}_1},\ldots,c_{\mathbf{k}_{\mathcal{Q}}})^{\mathrm{T}}$ and $f\in\widetilde{\mathcal{S}}_{N,K}\left(\varphi\right)$ as defined in (\ref{3.1}). Then the following identity holds: $$\left<C,\varPsi_j C\right>=\left|\mathcal{R}_\mathbf{p}\left(f\left(\mathbf{x}_j\right)\right)\right|^2.$$ Similarly,
    \begin{align}
       \nonumber\left<C,\mathbb{E}\left(\varPsi_j\right) C\right>&=\sum_{l_m}\sum_{l_n}c_{\mathbf{k}_{l_m}}\overline{c_{\mathbf{k}_{l_n}}}\overline{(\mathbb{E}\left(\varPsi_j\right))_{l_m,l_n}}\\
       \nonumber&=\sum_{l_m}\sum_{l_n}c_{\mathbf{k}_{l_m}}\overline{c_{\mathbf{k}_{l_n}}}\left<\sqrt{\xi(\mathbf{x})}\mathcal{R}_{\mathbf{p}}(\varphi(\mathbf{x}-\mathbf{k}_{l_m})),\sqrt{\xi(\mathbf{x})}\mathcal{R}_{\mathbf{p}}(\varphi(\mathbf{x}-\mathbf{k}_{l_n}))\right>\\
       \nonumber&=\left<\sqrt{\xi(\mathbf{x})}\sum_{l_m}c_{\mathbf{k}_{l_m}}\mathcal{R}_{\mathbf{p}}(\varphi(\mathbf{x}-\mathbf{k}_{l_m})),\sqrt{\xi(\mathbf{x})}\sum_{l_n}c_{\mathbf{k}_{l_n}}\mathcal{R}_{\mathbf{p}}(\varphi(\mathbf{x}-\mathbf{k}_{l_n}))\right>\\
       \nonumber&=\left\|\sqrt{\xi(\mathbf{x})}\mathcal{R}_{\mathbf{p}}(f(\mathbf{x}))\right\|_{L^2(\mathcal{E}_K)}^2.
    \end{align}
    Thus, it follows from assumption $\mathbf{(A.1)}$, (\ref{3.6}) and $Y_j\left(\mathcal{R}_{\mathbf{p}}(f)\right)$ in (\ref{3.4}) that
    \begin{align}
        \nonumber &\underset{{f\in \widetilde{\mathcal{S}}_{N,K}\left(\varphi\right)}}{sup}\left|\sum_{j=1}^{n}Y_j\left(\mathcal{R}_{\mathbf{p}}(f)\right)\right|\\ \nonumber&=\underset{{f\in \widetilde{\mathcal{S}}_{N,K}\left(\varphi\right)}}{sup}\left|\sum_{j=1}^{n}\left(\left|\mathcal{R}_\mathbf{p}\left(f\left(\mathbf{x}_{j}\right)\right)\right|^2-\int_{\mathcal{E}_K}\xi\left(\mathbf{x}\right) \left|\mathcal{R}_\mathbf{p}\left(f\left(\mathbf{x}\right)\right)\right|^2 d\mathbf{x}\right)\right|\\
        \nonumber&=\underset{\frac{1}{M_2}\leqslant\left\|C\right\|_{\ell^2}\leqslant\frac{1}{m_2}}{sup}\left|\sum_{j=1}^{n}\big( \left< C,\varPsi_jC \right> - \left< C,\mathbb{E}\left(\varPsi_j\right)C \right>\big)\right|\\
        \nonumber&=\underset{\frac{1}{M_2}\leqslant\left\|C\right\|_{\ell^2}\leqslant\frac{1}{m_2}}{sup}\left\|C\right\|_{\ell^2}^2\left|\sum_{j=1}^{n}\left(\frac{\left<C,X_j C\right>}{C^{\mathrm{T}}C}\right)\right|\\
        \nonumber&\leqslant\frac{1}{m_2^2}\lambda_{max}\left(\sum_{j=1}^n X_j\right),
    \end{align}
    where $\lambda_{max}$ stands for the largest eigenvalue of a self-adjoint matrix.
    \par Next, we estimate $\left\|X_j\right\|$. By $f\in\widetilde{\mathcal{S}}_{N,K}\left(\varphi\right)$ in (\ref{3.1}), along with the inequalities (\ref{3.2}) and (\ref{3.3}), we deduce that
    \begin{equation}
        \nonumber\left<C,\varPsi_j C\right>=\left|\mathcal{R}_\mathbf{p}\left(f\left(x_j\right)\right)\right|^2\leqslant \left\|\mathcal{R}_{\mathbf{p}}\left(f\right)\right\|_{L^{\infty}{\left(\mathcal{E}_K\right)}}^2\leqslant4K^2C_{\theta}^2\frac{C_{\varphi}^2}{m_{2}^2}
    \end{equation}
    and
    \begin{equation}
        \left<C,\mathbb{E}\left(\varPsi_j\right) C\right>=\left\|\sqrt{\xi(\mathbf{x})}\mathcal{R}_{\mathbf{p}}(f(\mathbf{x}))\right\|_{L^2(\mathcal{E}_K)}^2\leqslant 4K^2 C_{\theta}^2C_{\xi,u}\left\|f\right\|_{L^2\left(\mathcal{E}_K\right)}^2=4K^2 C_{\theta}^2C_{\xi,u}.\label{NN3.7}
    \end{equation}
    Then by $X_j$ in (\ref{3.6}), we obtain the following estimate
    \begin{equation}
        \nonumber\left|<C,X_j C>\right|=\left|\left<C,\varPsi_j C\right>-\left<C,\mathbb{E}\left(\varPsi_j\right) C\right>\right|\leqslant4K^2C_{\theta}^2\left(\frac{C_{\varphi}^2}{m_{2}^2}+C_{\xi,u}\right).
    \end{equation}
    Therefore, by $\left\|f\right\|_{L^2\left(\mathcal{E}_K\right)}=1$ and $\frac{1}{M_2}\leqslant\left\|C\right\|_{\ell^2}\leqslant\frac{1}{m_2}$, we conclude that
    \begin{equation}
        \nonumber\left\|X_j\right\|\leqslant4K^2C_{\theta}^2 M_2^2\left(\frac{C_{\varphi}^2}{m_{2}^2}+C_{\xi,u}\right).
    \end{equation}
    \par Finally, we estimate $\sigma^2=\left\|\sum_{j=1}^{n}\mathbb{E}\left(X_{j}^2\right)\right\|$ in Lemma \ref{lem3.2}. Actually, from (\ref{3.6}), we obtain 
    \begin{equation}
        \mathbb{E}\left(X_{j}^2\right)=\mathbb{E}\left(\varPsi_{j}^2\right)-\left[\mathbb{E}\left(\varPsi_j\right)\right]^2\leqslant\mathbb{E}\left(\varPsi_{j}^2\right),\label{New3.7}
    \end{equation}
    where
    \begin{align}
        \nonumber\left(\varPsi_{j}^2\right)_{l_m,l_n}&=\sum_{s=1}^{\mathcal{Q}}\left(\varPsi_j\right)_{l_m,l_s}\left(\varPsi_j\right)_{l_s,l_n}\\
        \nonumber&=\sum_{s=1}^{\mathcal{Q}}\mathcal{R}_{\mathbf{p}}(\varphi(\mathbf{x}_j-\mathbf{k}_{l_m}))\overline{\mathcal{R}_{\mathbf{p}}(\varphi(\mathbf{x}_j-\mathbf{k}_{l_s}))}\mathcal{R}_{\mathbf{p}}(\varphi(\mathbf{x}_j-\mathbf{k}_{l_s}))\overline{\mathcal{R}_{\mathbf{p}}(\varphi(\mathbf{x}_j-\mathbf{k}_{l_n}))}\\
        \nonumber&=\sum_{s=1}^{\mathcal{Q}}\left|\mathcal{R}_{\mathbf{p}}(\varphi(\mathbf{x}_j-\mathbf{k}_{l_s}))\right|^2\left(\varPsi_j\right)_{l_m,l_n}.  
    \end{align}
    Moreover, we can see that
\begin{align}
 \nonumber&\sum_{s=1}^{\mathcal{Q}}\left|\mathcal{R}_{\mathbf{p}}(\varphi(\mathbf{x}_j-\mathbf{k}_{l_s}))\right|^2\\
 \nonumber&=\sum_{s=1}^{\mathcal{Q}}\left|\int_{-KC_{\theta}}^{KC_{\theta}}\varphi\left(\mathbf{p}\left(\mathbf{x}_j-\mathbf{k}_{l_s}\right)\cos\theta-s\sin\theta,\mathbf{p}\left(\mathbf{x}_j-\mathbf{k}_{l_s}\right)\sin\theta+s\cos\theta\right)ds\right|^2\\
 \nonumber&\leqslant 2KC_{\theta}\int_{-KC_{\theta}}^{KC_{\theta}}\sum_{s=1}^{\mathcal{Q}}\left|\varphi\left(\mathbf{p}\left(\mathbf{x}_j-\mathbf{k}_{l_s}\right)\cos\theta-s\sin\theta,\mathbf{p}\left(\mathbf{x}_j-\mathbf{k}_{l_s}\right)\sin\theta+s\cos\theta\right)\right|^2ds\\
\nonumber&\leqslant 2KC_{\theta}C_{\varphi}^{2}\int_{-KC_{\theta}}^{KC_{\theta}}1ds\\
\nonumber&\leqslant 4K^2C_{\theta}^2C_{\varphi}^2,
\end{align}
where the first and second inequality are derived from Cauchy-Schwarz inequality and $C_{\varphi}=sup_{\mathbf{x}\in \mathcal{E}_K}\sum_{l=1}^{\mathcal{Q}}{\left|\varphi\left(\mathbf{x}-\mathbf{k}_l\right)\right|}$ in Lemma \ref{lem3.1}. By (\ref{3.6}), (\ref{NN3.7}) and (\ref{New3.7}), we obtain
\begin{align}
\nonumber\sigma^2&=\left\|\sum_{j=1}^{n}\mathbb{E}\left(X_{j}^2\right)\right\|\leqslant\left\|\sum_{j=1}^{n}\mathbb{E}\left(\varPsi_{j}^2\right)\right\|\leqslant4K^2C_{\theta}^2C_{\varphi}^2\left\|\sum_{j=1}^{n}\mathbb{E}\left(\varPsi_{j}\right)\right\|\leqslant16K^4C_{\theta}^4C_{\varphi}^2 M_2^2C_{\xi,u}n.
\end{align}
The lemma follows directly from the matrix Bernstein inequality presented in Lemma \ref{lem3.2}, we conclude that
\begin{align}
        \mathbb{P}\left(\underset{{f\in \widetilde{\mathcal{S}}_{N,K}\left(\varphi\right)}}{sup}\left|\sum_{j=1}^{n}Y_j\left(\mathcal{R}_{\mathbf{p}}(f)\right)\right|\geqslant\lambda\right)\nonumber&\leqslant\mathbb{P}\left(\lambda_{max}\left(\sum_{j=1}^n X_j\right)\geqslant m_2^2\lambda\right)\\
        \nonumber&\leqslant\mathcal{Q}exp\left(-\frac{m_2^4\lambda^2 }{8K^2C_{\theta}^2M_2^2\left[4K^2C_{\theta}^2C_{\varphi}^2C_{\xi,u}n+\frac{\lambda \left(C_{\varphi}^2+C_{\xi,u}m_2^2\right)}{3}\right]}\right).
    \end{align}
\end{proof}
\par In the  following theorem, we will focus on the problem of Radon random sampling inequality which will require the linear independence of $\left\{\mathcal{R}_\mathbf{p}\varphi\left(\cdot-\mathbf{pk}_l\right)\right\}_{l=1}^{\mathcal{Q}}$.

\begin{thm}
    Suppose that $X=\left\{\mathbf{x}_j\right\}_{j=1}^n$ is a sequence of independent \label{thm3.4}random variables derived from a general probability distribution on $\mathcal{E}_K$, with the density function $\xi$ satisfying assumption $\mathbf{(A.2)}$. And the sequence $\left\{\mathcal{R}_\mathbf{p}\varphi\left(\cdot-\mathbf{pk}_l\right)\right\}_{l=1}^{\mathcal{Q}}$ is linearly independent, where $E=\left[\lceil -N-K\rceil,\lfloor K+N\rfloor\right]^{2}\cap \mathbb{Z}^2$ and $\mathcal{Q}$ is the cardinality of $E$. Let $C_{\theta}=\cos\theta+\sin\theta$ with $\theta\in\left[0,\pi/2\right)$. Then for multi-angles $\theta$, there exist constants $C_{1,\mathbf{p}}$, $C_{2,\mathbf{p}}$ satisfying $0<C_{1,\mathbf{p}}\leqslant C_{2,\mathbf{p}}<\infty$ and a constant $\gamma$ satisfying $\frac{2KC_{1,\mathbf{p}}C_{\xi,l}}{M_2^2}>\gamma>0$ such that the sampling inequality
\begin{align}
    \nonumber\left[\frac{2KC_{1,\mathbf{p}}C_{\xi,l}}{M_2^2}-\gamma\right]\left\|f\right\|_{L^2\left(\mathcal{E}_K\right)}^2&\leqslant\frac{1}{n}\sum_{j=1}^n\left|\mathcal{R}_\mathbf{p}\left(f\left(\mathbf{x}_j\right)\right)\right|^2\leqslant \left[\frac{2\sqrt{2}KC_{2,\mathbf{p}}C_{\xi,u}}{m_2^2}+\gamma\right]\left\|f\right\|_{L^2\left(\mathcal{E}_K\right)}^2
\end{align}

    \noindent holds with the probability at least
    \begin{equation}
         1-\epsilon_{\mathcal{Q}}:=1-\mathcal{Q}exp\left(-\frac{nm_2^4\gamma^2 }{8K^2C_{\theta}^2M_2^2\left[4K^2C_{\theta}^2C_{\varphi}^2C_{\xi,u}+\frac{\gamma\left(C_{\varphi}^2+C_{\xi,u}m_2^2\right)}{3}\right]}\right).\label{3.11}
    \end{equation}
\end{thm}
\begin{proof}
    For $f\in \mathcal{S} _{N,K}\left(\varphi\right)$ in (\ref{1.2}), we note that $f$ fulfills the sampling inequality if and only if $f/{\left\|f\right\|_{L^2\left(\mathcal{E}_K\right)}}$ satisfies the sampling inequality as well. Let $m\left(\mathcal{R}_\mathbf{p}\left(f\right)\right)=\int_{\mathcal{E}_K}\xi\left(\mathbf{x}\right)\left|\mathcal{R}_\mathbf{p}\left(f\left(\mathbf{x}\right) \right)\right|^2d\mathbf{x}$, where $f\in\widetilde{\mathcal{S}}_{N,K}\left(\varphi\right)$ in (\ref{3.1}). We define the event
    \begin{equation}
        \nonumber\mathcal{G}_f=\left\{\underset{f\in\widetilde{\mathcal{S}}_{N,K}\left(\varphi\right)}{sup}\left|\sum_{j=1}^n Y_j\left(\mathcal{R}_\mathbf{p}\left(f\right)\right)\right|\geqslant n\gamma\right\}.
    \end{equation}
    Its complement corresponds to the event
    \begin{equation}
        \overline{\mathcal{G}}_f=nm\left(\mathcal{R}_\mathbf{p}\left(f\right)\right)-n\gamma\leqslant\sum_{j=1}^n\left|\mathcal{R}_\mathbf{p}\left(f\left(\mathbf{x}_j\right)\right)\right|^2\leqslant n\gamma+nm\left(\mathcal{R}_\mathbf{p}\left(f\right)\right).\label{3.8}
    \end{equation}
    As in (\ref{2.1}) and (\ref{2.5}), there exists sequence $\left\{c_{\mathbf{k}_{l}},l=1,\ldots,\mathcal{Q}\right\}$ such that 
    \begin{equation}
        \nonumber f\left(\mathbf{x}\right)=\sum_{l=1}^{\mathcal{Q}}c_{\mathbf{k}_{l}}\varphi\left(\mathbf{x}-\mathbf{k}_l\right),\; \forall\;\mathbf{x}\in \mathcal{E}_K
    \end{equation}
   and consequently,
   \begin{equation}
   \nonumber \mathcal{R}_\mathbf{p}\left(f\left(\mathbf{x}\right)\right)(t)=\sum_{l=1}^{\mathcal{Q}}c_{\mathbf{k}_{l}}\mathcal{R}_\mathbf{p}\left(\varphi\left(\mathbf{x}-\mathbf{k}_l\right)\right)(t).
    \end{equation}
    \indent Next, let $\mathbf{x}=\left(x,y\right)$ and $\mathbf{p}=\left(\cos\theta,\sin\theta\right)$, we will estimate $m\left(\mathcal{R}_\mathbf{p}\left(f\right)\right)$ with $t=\mathbf{px}$, we first consider $\left\|\mathcal{R}_{\mathbf{p}}\left(f\right)\right\|_{L^2{\left(\mathcal{E}_K\right)}}^{2}$.
    \begin{align}
        \nonumber\int_{\mathcal{E}_K}\left|\mathcal{R}_\mathbf{p}\left(f\left(\mathbf{x}\right)\right)\right|^2d\mathbf{x}&=\int_{\mathcal{E}_K}\left|\sum_{l=1}^{\mathcal{Q}}c_{\mathbf{k}_{l}}\mathcal{R}_\mathbf{p}\left(\varphi\left(\mathbf{x}-\mathbf{k}_{l}\right)\right)\right|^2 d\mathbf{x}\\
    
        &=\int_{-K}^K\int_{-K}^K\left|\sum_{l=1}^{\mathcal{Q}}c_{\mathbf{k}_{l}}\mathcal{R}_\mathbf{p}\varphi\left(x\cos\theta+y\sin\theta-\mathbf{pk}_{l}\right)\right|^2 dxdy,\label{3.9}
    \end{align}
    where the second equality is derived from $\mathcal{R}_{\mathbf{p}}\left(\varphi\left(\mathbf{x}\right)\right)\left(t\right)=\mathcal{R}_{\mathbf{p}}\varphi\left(\mathbf{px}\right)$. Due to $t=\mathbf{px}$ and $s=-x \sin\theta+y\cos\theta$, from Lemma \ref{lem3.1}, we know that
    \[
    t,s\in
    \begin{cases}
        \left[-K\left(\sin\theta+\cos\theta\right),K\left(\sin\theta+\cos\theta\right)\right],&\theta\in\left[0,\pi/2\right),\\
        \left[-K\left(\sin\theta-\cos\theta\right),K\left(\sin\theta-\cos\theta\right)\right],&\theta\in\left[\pi/2,\pi\right),\\
        \left[-K\left(-\sin\theta-\cos\theta\right),-K\left(\sin\theta+\cos\theta\right)\right],&\theta\in\left[\pi,3\pi/2\right),\\
        \left[-K\left(\cos\theta-\sin\theta\right),K\left(\cos\theta-\sin\theta\right)\right],&\theta\in\left[3\pi/2,2\pi\right).\\
    \end{cases}
    \]
    Firstly, we continue to calculate (\ref{3.9}). We consider $\theta\in\left[0,\pi/2\right)$,
    \begin{align}
        \nonumber &\int_{-K}^K\int_{-K}^K\left|\sum_{l=1}^{\mathcal{Q}}c_{\mathbf{k}_{l}}\mathcal{R}_\mathbf{p}\varphi\left(x\cos\theta+y\sin\theta-\mathbf{pk}_{l}\right)\right|^2 dxdy\\
        \nonumber &=\int_{-KC_{\theta}}^{KC_{\theta}}\int_{-KC_{\theta}}^{KC_{\theta}}\left|\sum_{l=1}^{\mathcal{Q}}c_{\mathbf{k}_{l}}\mathcal{R}_\mathbf{p}\varphi\left(t-\mathbf{pk}_{l}\right) \right|^2 dtds\\
        \nonumber &=2KC_{\theta}\int_{-KC_{\theta}}^{KC_{\theta}}\left|\sum_{l=1}^{\mathcal{Q}} c_{\mathbf{k}_{l}}\mathcal{R}_\mathbf{p}\varphi\left(t-\mathbf{pk}_{l}\right) \right|^2 dt,
    \end{align}
    \noindent then by the independence of $\left\{\mathcal{R}_\mathbf{p}\varphi\left(\cdot-\mathbf{pk}_{l}\right),l=1,\ldots,\mathcal{Q}\right\}$, we suppose that there exist $0<C_{1,\mathbf{p}}\leqslant C_{2,\mathbf{p}}<\infty$ such that
    \begin{equation}
C_{1,\mathbf{p}}\sum_{l=1}^{\mathcal{Q}}\left|c_{\mathbf{k}_{l}}\right|^2\leqslant \int_{-KC_{\theta}}^{KC_{\theta}}\left|\sum_{l=1}^{\mathcal{Q}}c_{\mathbf{k}_{l}}\mathcal{R}_\mathbf{p}\varphi\left(t-\mathbf{pk}_{l}\right)\right|^2 dt \leqslant C_{2,\mathbf{p}}\sum_{l=1}^{\mathcal{Q}}\left|c_{\mathbf{k}_{l}}\right|^2.\label{3.10}
    \end{equation}
    Therefore, it follows from assumption $\mathbf{(A.2)}$, (\ref{3.10}) and the definition of $m\left(\mathcal{R}_\mathbf{p}\left(f\right)\right)$ that 
    \begin{align} \nonumber2KC_{1,\mathbf{p}}C_{\xi,l}C_{\theta}\sum_{l=1}^{\mathcal{Q}}\left|c_{\mathbf{k}_{l}}\right|^2\leqslant\int_{\mathcal{E}_K}\xi\left(\mathbf{x}\right)\left|\mathcal{R}_\mathbf{p}\left(f\left(\mathbf{x}\right)\right)\right|^2d\mathbf{x}\leqslant2KC_{2,\mathbf{p}}C_{\xi,u}C_{\theta}\sum_{l=1}^{\mathcal{Q}}\left|c_{\mathbf{k}_{l}}\right|^2,
    \end{align}
    where $C_{\theta}=cos\theta+sin\theta,\,\theta\in\left[0,\pi/2\right)$. Then, we summarize the four cases of the angle $\theta$.\\
    \indent When $\theta\in\left[0,\pi/2\right)$, the following inequality holds,
    \begin{align}
    \nonumber2KC_{1,\mathbf{p}}C_{\xi,l}\left(\sin\theta+\cos\theta\right)\sum_{l=1}^{\mathcal{Q}}\left|c_{\mathbf{k}_{l}}\right|^2 &\leqslant\int_{\mathcal{E}_K}\xi\left(\mathbf{x}\right)\left|\mathcal{R}_\mathbf{p}\left(f\left(\mathbf{x}\right)\right)\right|^2d\mathbf{x}\\
    \nonumber&\leqslant2KC_{2,\mathbf{p}}C_{\xi,u}\left(\sin\theta+\cos\theta\right)\sum_{l=1}^{\mathcal{Q}}\left|c_{\mathbf{k}_{l}}\right|^2.
    \end{align}
    \indent when $\theta\in\left[\pi/2,\pi\right)$,  the following inequality holds,
    \begin{align}
        \nonumber2KC_{1,\mathbf{p}}C_{\xi,l}\left(\sin\theta-\cos\theta\right)\sum_{l=1}^{\mathcal{Q}}\left|c_{\mathbf{k}_{l}}\right|^2 &\leqslant\int_{\mathcal{E}_K}\xi\left(\mathbf{x}\right)\left|\mathcal{R}_\mathbf{p}\left(f\left(\mathbf{x}\right)\right)\right|^2d\mathbf{x}\\
        \nonumber&\leqslant2KC_{2,\mathbf{p}}C_{\xi,u}\left(\sin\theta-\cos\theta\right)\sum_{l=1}^{\mathcal{Q}}\left|c_{\mathbf{k}_{l}}\right|^2.
    \end{align}
    \indent When $\theta\in\left[\pi,3\pi/2\right)$,  the following inequality holds,
    \begin{align}
        \nonumber2KC_{1,\mathbf{p}}C_{\xi,l}\left(-\sin\theta-\cos\theta\right)\sum_{l=1}^{\mathcal{Q}}\left|c_{\mathbf{k}_{l}}\right|^2&\leqslant\int_{\mathcal{E}_K}\xi\left(\mathbf{x}\right)\left|\mathcal{R}_\mathbf{p}\left(f\left(\mathbf{x}\right)\right)\right|^2d\mathbf{x}\\
        \nonumber&\leqslant2KC_{2,\mathbf{p}}C_{\xi,u}\left(-\sin\theta-\cos\theta\right)\sum_{l=1}^{\mathcal{Q}}\left|c_{\mathbf{k}_{l}}\right|^2.
    \end{align}
    \indent When $\theta\in\left[3\pi/2,2\pi\right)$,  the following inequality holds,
    \begin{align}
        \nonumber2KC_{1,\mathbf{p}}C_{\xi,l}\left(\cos\theta-\sin\theta\right)\sum_{l=1}^{\mathcal{Q}}\left|c_{\mathbf{k}_{l}}\right|^2&\leqslant\int_{\mathcal{E}_K}\xi\left(\mathbf{x}\right)\left|\mathcal{R}_\mathbf{p}\left(f\left(\mathbf{x}\right)\right)\right|^2d\mathbf{x}\\
        \nonumber&\leqslant2KC_{2,\mathbf{p}}C_{\xi,u}\left(\cos\theta-\sin\theta\right)\sum_{l=1}^{\mathcal{Q}}\left|c_{\mathbf{k}_{l}}\right|^2.
    \end{align}
    Observing these four cases, we can know that the trigonometric function between the brackets at both ends of the inequality can take the minimum value 1 and the maximum value $\sqrt{2}$ in the range of $\theta$.\\ 
    \indent We denote $C=\{c_{\mathbf{k}_l}\}_{l=1}^{\mathcal{Q}}$. By $f\in \widetilde{\mathcal{S}}_{N,K}\left(\varphi\right)$ in (\ref{3.1}) and $(\mathbf{A.1})$, we have $\left\|f\right\|_{L^2\left(\mathcal{E}_K\right)}=1$ and $\frac{1}{M_2}\leqslant\left\|C\right\|_{\ell^2}\leqslant\frac{1}{m_2}$, it follows that $\frac{2KC_{1,\mathbf{p}}C_{\xi,l}}{M_2^2}\leqslant m\left(\mathcal{R}_\mathbf{p}\left(f\right)\right)\leqslant \frac{2\sqrt{2}KC_{2,\mathbf{p}}C_{\xi,u}}{m_2^2}$. For $f\in \mathcal{S} _{N,K}\left(\varphi\right)$ in (\ref{1.2}), we define the event
    \begin{equation}
\begin{split}
\nonumber\widetilde{\mathcal{G}}_f = \Biggl\{ 
\left[\frac{2KC_{1,\mathbf{p}}C_{\xi,l}}{M_2^2} - \gamma\right] \|f\|_{L^2(\mathcal{E}_K)}^2 
    &\leqslant \frac{1}{n}\sum_{j=1}^n \bigl|\mathcal{R}_\mathbf{p}(f(\mathbf{x}_j))\bigr|^2 \leqslant \left[\frac{2\sqrt{2}KC_{2,\mathbf{p}}C_{\xi,u}}{m_2^2} + \gamma\right] \|f\|_{L^2(\mathcal{E}_K)}^2 
\Biggr\}.
\end{split}
\end{equation}
    By the above discussion, we conclude that $\overline{\mathcal{G}}_f\subseteq \widetilde{\mathcal{G}}_f$, where $\overline{\mathcal{G}}_f$ is defined in (\ref{3.8}). By Lemma \ref{lem3.3}, the sampling inequality is consistently satisfied for $f\in \mathcal{S} _{N,K}\left(\varphi\right)$ with the probability
    \begin{align}
    \nonumber\mathbb{P}\left(\widetilde{\mathcal{G}}_f\right)&\geqslant\mathbb{P}\left(\overline{\mathcal{G}}_f\right)=1-\mathbb{P}\left(\mathcal{G}_f \right)
    =1-\epsilon_{\mathcal{Q}}.
    \end{align}
\end{proof}  
\begin{rem}
    Notice that the sequence $\left\{\mathcal{R}_\mathbf{p}\varphi\left(\cdot-\mathbf{pk}_l\right),l=1,\ldots,\mathcal{Q}\right\}$ needs to be linearly independent. We can give an example. Let $\varphi\left(x,y\right)=B_2\left(x+1\right)B_2\left(y+1\right)$, where $B_2\left(x\right)=\chi_{\left(0,1\right]}\ast\chi_{\left(0,1\right]}\left(x\right)$ and $\ast$ is a convolution operation. By Proposition 3.3 in \cite{2023Determination}, we know that $\left\{\mathcal{R}_\mathbf{p}\varphi\left(\cdot-\mathbf{pk}_l\right)\right\}_{l=1}^{\mathcal{Q}}$ is linearly independent. 
\end{rem}

\section{Reconstruction from Radon random sampling in $\mathcal{S} _{N,K}\left(\varphi\right)$}
In this section, we present a sufficient condition of reconstruction in Lemma \ref{lem3.7}, which can be derived from Theorem \ref{thm2.1}. It will be used to demonstrate the major result in Theorem \ref{thm3.8}. Let $\mathcal{Q}$ denote the cardinality of $E$ in (\ref{2.2}). We define the sampling matrix 
\begin{equation}
    U=\left(U_1,\ldots,U_{\mathcal{Q}}\right)\;, U_l=\left(\mathcal{R}_\mathbf{p}\left(\varphi\left(\mathbf{x}_1-\mathbf{k}_l\right)\right),\ldots,\mathcal{R}_\mathbf{p}\left(\varphi\left(\mathbf{x}_n -\mathbf{k}_l\right)\right)\right)^{\mathrm{T}}. \label{3.21}
\end{equation}
\begin{lem}
    For $\sigma>0$, let $X=\left\{\mathbf{x}_j,j=1,\ldots,n\right\}\subseteq \mathcal{E}_K$ be 
 a sampling set that satisfies\label{lem3.7} 
    \begin{equation} \sum_{j=1}^{n}\left|\mathcal{R}_\mathbf{p}f\left(\mathbf{x}_j\right)\right|^2\geqslant\sigma\left\|C\right\|_{\ell^2}^2 .\label{3.12}
    \end{equation} Then for any $f\in \mathcal{S} _{N,K}\left(\varphi\right)$, there exist reconstruction functions $\left\{\varUpsilon_j\left(\mathbf{x}\right)\right\}_{j=1}^n$ such that 
    \begin{equation}
        \nonumber f\left(\mathbf{x}\right)=\sum_{j=1}^{n}\mathcal{R}_\mathbf{p}\left(f\left(\mathbf{x}_j\right)\right)\varUpsilon_j\left(\mathbf{x}\right),\; \forall\;\mathbf{x}\in \mathcal{E}_K,
    \end{equation}
    where $\varUpsilon_j\left(\mathbf{x}\right)=\sum_{l=1}^{\mathcal{Q}}\sum_{i=1}^{\mathcal{Q}}\mathcal{R}_\mathbf{p}\left(\varphi\left(\mathbf{x}_j-\mathbf{k}_i\right)\right)\left(\left(U^{\mathrm{T}}U\right)^{-1}\right)_{il}\varphi\left(\mathbf{x}-\mathbf{k}_l\right)$ and $U$ is define in (\ref{3.21}).
\end{lem}
\begin{proof}
    By $f\left(\mathbf{x}\right)=\sum_{l=1}^{\mathcal{Q}}c_{\mathbf{k}_{l}}\varphi\left(\mathbf{x}-\mathbf{k}_l\right)$ in (\ref{2.1}) and $t=\mathbf{px}$, we obtain 
    \begin{equation}
        \nonumber \mathcal{R}_\mathbf{p}\left(f\left(\mathbf{x}\right)\right)=\sum_{l=1}^{\mathcal{Q}}c_{\mathbf{k}_{l}}\mathcal{R}_\mathbf{p}\left(\varphi\left(\mathbf{x}-\mathbf{k}_l\right)\right),\;\forall\;\mathbf{x}\in \mathcal{E}_K,
    \end{equation}
    where $\mathcal{Q}$ is the cardinality of $E$ in (\ref{2.2}).
    Let $Y=\left(\mathcal{R}_\mathbf{p}\left(f\left(\mathbf{x}_1\right)\right),\ldots,\mathcal{R}_\mathbf{p}\left(f\left(\mathbf{x}_n\right)\right)\right)^{\mathrm{T}} $ and $C=\left(c_{\mathbf{k}_{1}},\ldots,c_{\mathbf{k}_{\mathcal{Q}}}\right)^{\mathrm{T}}$. The matrix form can be written as follows:
\begin{equation}
    UC=Y,\label{3.13}
\end{equation}
where the matrix $U$ is defined in (\ref{3.21}). By (\ref{3.12}), one has $C^{\mathrm{T}} U^{\mathrm{T}} UC\geqslant\sigma\left\|C\right\|_{\ell^2}^2>0$, which implies that the matrix $U^{\mathrm{T}} U$ is invertible. By (\ref{3.13}), we know $C=\left(U^{\mathrm{T}} U\right)^{-1}U^{\mathrm{T}}Y$. Therefore, we obtain
    \begin{align}
        \nonumber f\left(\mathbf{x}\right)=C^{\mathrm{T}}\varPhi=Y^{\mathrm{T}} U\left(U^{\mathrm{T}} U\right)^{-1}\varPhi
        =Y^{\mathrm{T}} \varUpsilon\left(\mathbf{x}\right),
    \end{align}
    where  $\varUpsilon\left(\mathbf{x}\right)=U\left(U^{\mathrm{T}} U\right)^{-1}\varPhi=\left\{\varUpsilon_j\left(\mathbf{x}\right)\right\}_{j=1}^{n}$ and $\varPhi=\left(\varphi\left(\mathbf{x}-\mathbf{k}_{1}\right),\ldots,\varphi\left(\mathbf{x}-\mathbf{k}_{\mathcal{Q}}\right)\right)^{\mathrm{T}}$. Then $f\left(\mathbf{x}\right)$ can be reconstructed by
    \begin{equation}
        \nonumber f\left(\mathbf{x}\right)=\sum_{j=1}^n \mathcal{R}_\mathbf{p}\left(f\left(\mathbf{x}_j\right)\right)\varUpsilon_j\left(\mathbf{x}\right),
    \end{equation}
    where $\varUpsilon_j\left(\mathbf{x}\right)=\sum_{l=1}^{\mathcal{Q}}\sum_{i=1}^{\mathcal{Q}}\mathcal{R}_\mathbf{p}\left(\varphi\left(\mathbf{x}_j-\mathbf{k}_i\right)\right)\left(\left(U^{\mathrm{T}}U\right)^{-1}\right)_{il}\varphi\left(\mathbf{x}-\mathbf{k}_l\right).$
\end{proof}
The theorem that follows provides the formula for reconstructing all functions $f(\mathbf{x})\in \mathcal{S} _{N,K}\left(\varphi\right)$ with high probability.
\begin{thm}
    Let $\mathcal{S} _{N,K}\left(\varphi\right)$ be defined by (\ref{3.1})\label{thm3.8}, and let $X=\left\{\mathbf{x}_j\right\}_{j=1}^n$ be a set of independent random variables derived from the general probability distribution on $\mathcal{E}_K$ with the density function $\xi$ satisfying the assumption $\mathbf{(A.2)}$. The sequence $\left\{\mathcal{R}_\mathbf{p}\varphi\left(\cdot-\mathbf{pk}_l\right)\right\}_{l=1}^{\mathcal{Q}}$ is linearly independent. Then for $C_{1,\mathbf{p}}$, $C_{2,\mathbf{p}}$ satisfying $0<C_{1,\mathbf{p}}\leqslant C_{2,\mathbf{p}}<\infty$ and $\gamma$ satisfying $0<\gamma<\frac{2KC_{1,\mathbf{p}}C_{\xi,l}}{M_2^2}$, there exist reconstruction functions $\left\{\varUpsilon_j\left(\mathbf{x}\right)\right\}_{j=1}^n$ such that for all functions in (\ref{1.2}), the reconstruction formula
    \begin{equation}
        f\left(\mathbf{x}\right)=\sum_{j=1}^n \mathcal{R}_\mathbf{p}\left(f\left(\mathbf{x}_j\right)\right)\varUpsilon_j\left(\mathbf{x}\right),\;\forall\; \mathbf{x}\in \mathcal{E}_K\label{3.14}
    \end{equation}
    holds with probability at least $1-\epsilon_{\mathcal{Q}}$, where 
    $$\epsilon_{\mathcal{Q}}=\mathcal{Q}exp\left(-\frac{nm_2^4\gamma^2 }{8K^2C_{\theta}^2M_2^2\left[4K^2C_{\theta}^2C_{\varphi}^2C_{\xi,u}+\frac{\gamma\left(C_{\varphi}^2+C_{\xi,u}m_2^2\right)}{3}\right]}\right)$$
   in (\ref{3.11}), $\varUpsilon_j\left(\mathbf{x}\right)=\sum_{l=1}^{\mathcal{Q}}\sum_{i=1}^{\mathcal{Q}}\mathcal{R}_\mathbf{p}\left(\varphi\left(\mathbf{x}_j-\mathbf{k}_i\right)\right)\left(\left(U^{\mathrm{T}}U\right)^{-1}\right)_{il}\varphi\left(\mathbf{x}-\mathbf{k}_l\right)$ and $U$ is defined in (\ref{3.21}). 
\end{thm}
\begin{proof}
    For any $0<\alpha_1<\alpha_2$, we define the following events
    \begin{align}
        \nonumber&\mathcal{I} _{f}=\left\{\frac{\alpha_1}{m_2^2}\left\|f\right\|_{L^2\left(\mathcal{E}_K\right)}^2\leqslant\sum_{j=1}^n \left|\mathcal{R}_\mathbf{p}\left(f\left(\mathbf{x}_j\right)\right)\right|^2\leqslant\frac{\alpha_2}{M_2^2}\left\|f\right\|_{L^2\left(\mathcal{E}_K\right)}^2,\,\forall f\in \mathcal{S} _{N,K}\left(\varphi\right)\right\},\\
        \nonumber&\mathcal{I} _c=\left\{\alpha_1\left\|C\right\|_{\ell^2}^2\leqslant\left\|UC\right\|_{\ell^2}^2\leqslant\alpha_2\left\|C\right\|_{\ell^2}^2\right\},\\
        \nonumber&\mathcal{I} _{r}=\left\{f\left(\mathbf{x}\right)=\sum_{j=1}^n \mathcal{R}_\mathbf{p}\left(f\left(\mathbf{x}_j\right)\right)\varUpsilon_j\left(\mathbf{x}\right),\forall f\in \mathcal{S} _{N,K}\left(\varphi\right)\right\}.
    \end{align}
    By $\mathbf{(A.1)}$, it is easy to verify that $\mathcal{I} _f\subseteq\mathcal{I} _c$. Then by Lemma \ref{lem3.7}, we obtain
    $\mathcal{I} _c\subseteq\mathcal{I} _r$. Thus, one has $\mathcal{I} _f\subseteq\mathcal{I} _c\subseteq\mathcal{I} _r$. By choosing $\alpha_1=n\left[\frac{2KC_{1,\mathbf{p}}C_{\xi,l}}{M_2^2}-\gamma\right]m_2^2$ and $\alpha_2=n\left[\frac{2\sqrt{2}KC_{2,\mathbf{p}}C_{\xi,l}}{m_2^2}+\gamma\right]M_2^2$, we obtain the sampling inequality in Theorem \ref{thm3.4}. Then for all $f\in \mathcal{S} _{N,K}\left(\varphi\right)$, the equation (\ref{3.14}) holds with probability at least $$\mathbb{P}\left(\mathcal{I} _r\right)\geqslant\mathbb{P}\left(\mathcal{I} _c\right)\geqslant\mathbb{P}\left(\mathcal{I} _f\right)\geqslant 1-\epsilon_{\mathcal{Q}}.$$   
\end{proof}

\section{Numerical Test}
    Motivated by the fact that shift-invariant spaces generated by box splines are used in \cite{Entezari20121532} to represent biomedical images in the continuous domain, we perform multiple experiments in a local shift-invariant signal space $\mathcal{S} _{N,K}\left(\varphi\right)$ formed by a positive definite box spline. Let $\varphi\left(x,y\right)=B_2\left(x+1\right)B_2\left(y+1\right)$, where $B_2\left(x\right)=\chi_{\left(0,1\right]}\ast\chi_{\left(0,1\right]}\left(x\right)$ and $\ast$ is a convolution operation. Due to supp$\left(B_2\left(x+1\right)\right)=[-1,1]$, we know supp$\left(\varphi\right)=\left[-1,1\right]^2$. Let $\mathcal{E}_K=\left[-\frac{1}{2},\frac{1}{2}\right]^2$, we can see that $E=\left[-1,1\right]^2\cap\mathbb{Z}^2$.

        

We choose $\mathbf{p}=\left(\cos\theta,\sin\theta\right)$ where $0<\theta<\pi/2$ and $\tan\theta>2$. Then by (\ref{1.1}), we have
    \[
    \mathcal{R}_{\mathbf{p}}\varphi\left(t\right)=
    \begin{cases}
        \frac{\left(\tan\theta-\frac{t}{\cos\theta}\right)\left[\left(\frac{t}{\cos\theta}-\tan\theta-\frac{3}{2}\right)^2+\frac{3}{4}\right]+1}{6\cos\theta \tan^2\theta},&t\in\left(\sin\theta,\cos\theta+\sin\theta\right)\\
        \frac{3\left(\frac{t}{\cos\theta}-\tan\theta\right)^2+\left(\frac{t}{\cos\theta}-\tan\theta\right)^3+3\tan\theta-\frac{3t}{\cos\theta}+1}{6\cos\theta \tan^2\theta},&t\in\left(\sin\theta-\cos\theta,\sin\theta\right)\\
        \frac{\tan\theta-\frac{x}{\cos\theta}}{\cos\theta \tan^2\theta},&t\in\left[\cos\theta,\sin\theta-\cos\theta\right]\\
        \frac{\frac{t}{\cos\theta}\left(\frac{2t^2}{\cos^2\theta}-\frac{6t}{\cos\theta}+3\right)+6\tan\theta-\frac{3t}{\cos\theta}-2}{6\cos\theta \tan^2\theta},&t\in\left[0,\cos\theta\right]\\
        \mathcal{R}_{\mathbf{p}}\varphi\left(-t\right),&t\in\left(-\sin\theta-\cos\theta,0\right]\\
        0,& otherwise.
    \end{cases}
    \]

 We have similar expression for $0<\tan\theta<1$ and $1<\tan\theta<2$. As shown in Figure \ref{Fig3}, we draw the image of $\varphi\left(\mathbf{x}-\mathbf{k}\right)$ with $\mathbf{k}=\left(1,-1\right)$ and $\mathcal{R}_\mathbf{p}\left(\varphi\left(\mathbf{x}-\mathbf{k}\right)\right)$ with $\mathbf{p}$ = [$\frac{5}{13}$, $\frac{12}{13}$]. 

 \begin{figure}[h]
\centering
\includegraphics[scale=0.35]{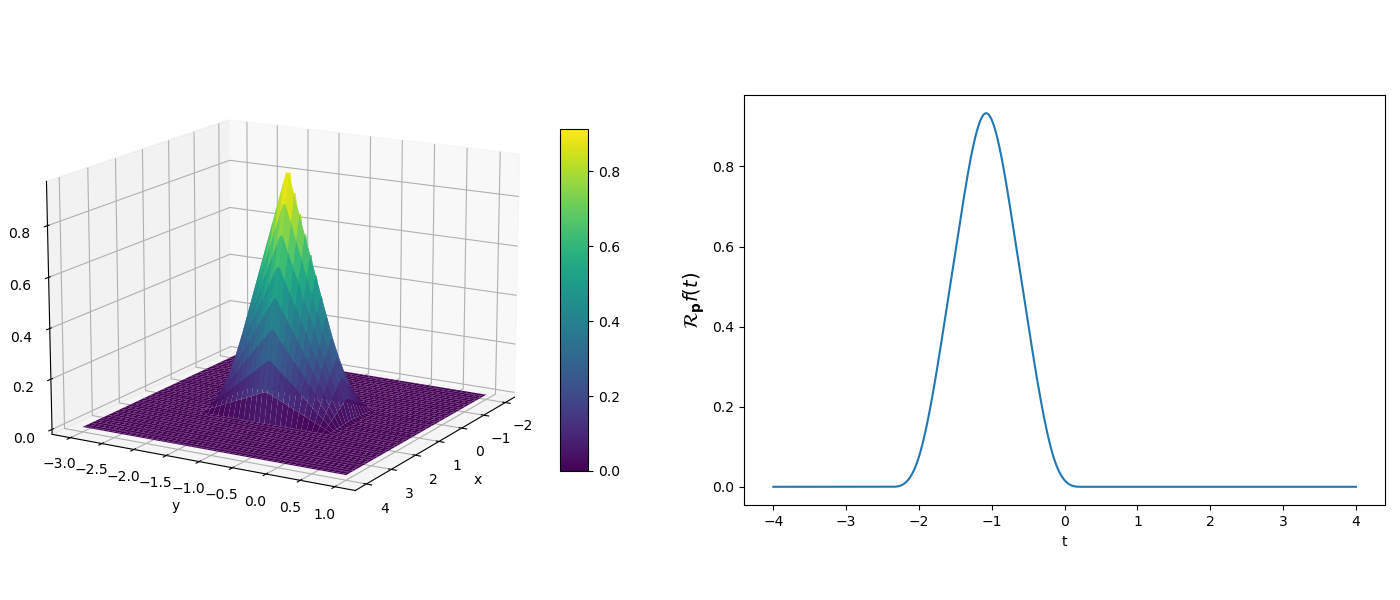}
\caption{Left: the plot of $\varphi\left(\mathbf{x}-\mathbf{k}\right)$ with $\mathbf{k}=\left(1,-1\right)$. Right: the plot of $\mathcal{R}_\mathbf{p}\left(\varphi\left(\mathbf{x}-\mathbf{k}\right)\right)$ with 
$\mathbf{p}$ = [$\frac{5}{13}$, $\frac{12}{13}$].}\label{Fig3}
\end{figure}
  Without bias, we choose $f\in \mathcal{S} _{N,K}\left(\varphi\right)$,$$f\left(\mathbf{x}\right)=\sum_{\mathbf{k}_l=\left(i,j\right)=\left\{-1,0,1\right\}^2}c_{\mathbf{k}_l}\varphi\left(\mathbf{x}-\mathbf{k}_l\right),\,\forall\, \mathbf{x}\in \mathcal{E}_K=\left[-\frac{1}{2},\frac{1}{2}\right]^2,$$ 
\begin{figure}
\centering
\includegraphics[scale=.52]{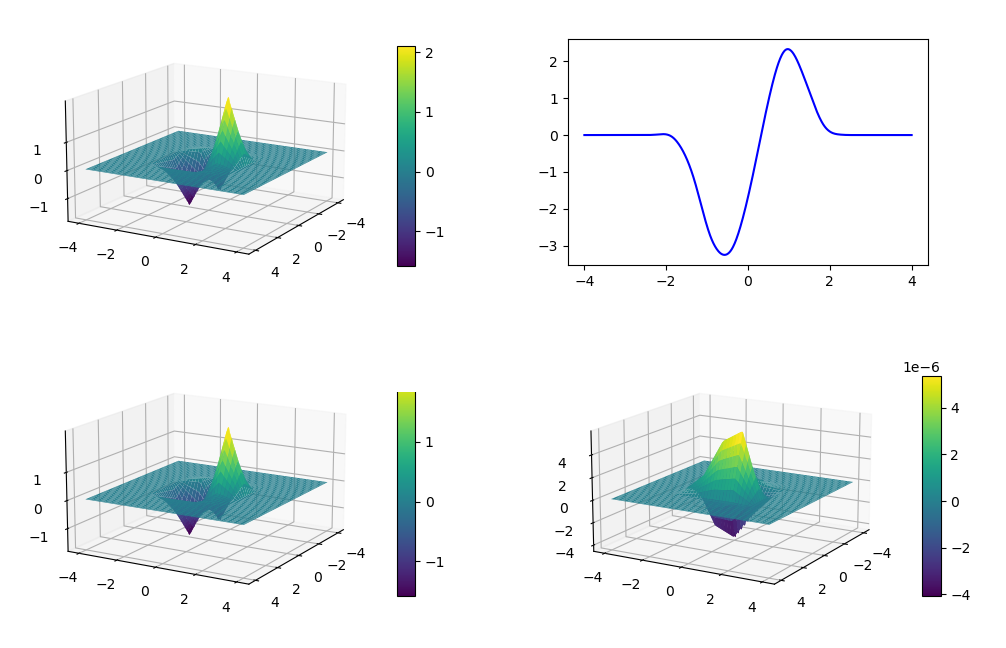}
\caption{Top left: the plot of function $f$. Top right:  Radon transform $\mathcal{R}_{\mathbf{p}}f$ with $\mathbf{p}=\left[\cos\left(\frac{5}{13}\pi\right),\sin\left(\frac{12}{13}\pi\right)\right]$. Bottom left: reconstruction version $\widetilde{f}$ of $f$. Bottom right: the plot of $f-\widetilde{f}$}\label{Fig4}
\end{figure}
\begin{figure}[h]
\centering
\includegraphics[scale=.45]{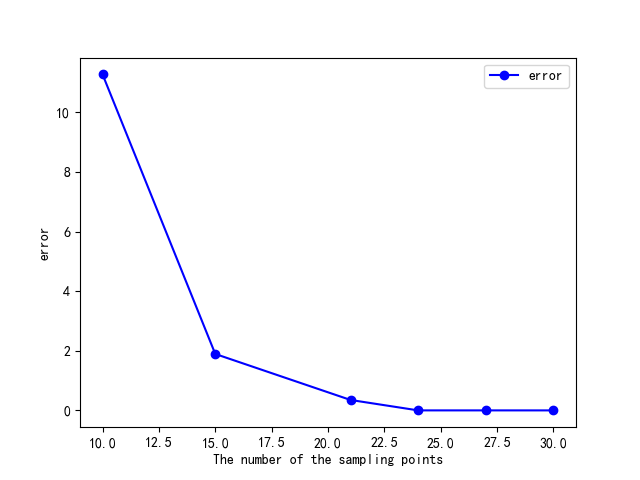}
\caption{Error analysis of varied sampling points}\label{Fig5}
\end{figure}
 \noindent where the coefficient matrix
\begin{equation}
    \nonumber C = \{ c_{i,j} \}_{i,j=0}^{2} = 
    \begin{pmatrix}
        0.1717 & -1.3467 & 0.1075 \\
        -1.7869 & -0.3373 & 2.4782 \\
        -0.8612 & -0.3645 & 0.2011
    \end{pmatrix}.
\end{equation} 
    Then by (\ref{2.5}) and $t=\mathbf{px}$, we have $$\mathcal{R}_{\mathbf{p}}\left(f\left(\mathbf{x}\right)\right)=\sum_{l=1}^{9}c_{\mathbf{k}_l}\mathcal{R}_{\mathbf{p}}\left(\varphi\left(\mathbf{x}-\mathbf{k}_l\right)\right),\,\forall\, \mathbf{x}\in \mathcal{E}_K=\left[-\frac{1}{2},\frac{1}{2}\right]^2,$$where $\left\{\mathbf{k}_{1},\ldots,\mathbf{k}_{9}\right\}=\left\{-1,0,1\right\}^2$ is arranged in the lexicographical order.\\
    \indent Next, we select 30 sampling points $\mathbf{x}$, which are uniformly distributed over the interval $\left[-\frac{1}{2},\frac{1}{2}\right]^2$. Given that $K=\frac{1}{2}$, $N=1$ and $\mathcal{Q}$ is the cardinality of $E$, it follows that the selection of sample points is reasonable. Then the sequence $\{{\widetilde{c}_{\mathbf{k}_l}}\}$ can be determined by (\ref{3.14}).  The following is the error calculation formula:
\begin{align}
    \nonumber error = \frac{\left\|\{c_{\mathbf{k}_l}-\widetilde{c}_{\mathbf{k}_l}\}_{l=1}^{9}\right\|_2}{\left\|\{c_{\mathbf{k}_l}\}_{l=1}^{9}\right\|_2}.
\end{align}

We find that the error gradually decreases as the number of sampling points increases and eventually stabilizes at the order of $10^{-6}$. The recovery of the function $f$ from Radon random samples is shown in Figure \ref{Fig4}. The error analysis is shown in Figure \ref{Fig5}.



\section{Conclusion}

\noindent We address the problem of signal reconstruction from Radon random samples in the local shift-invariant signal space. A critical aspect of this reconstruction process is the identification of a stable sampling set, which ensures that the original function can be accurately recovered. We prove that for a sufficiently large sampling set, there is a high probability that a random selection from a square domain with a general probability distribution will form a stable Radon sampling set. The randomness of the Radon samples allows for the successful application of our proposed reconstruction formula.\\

\appendix



\bibliographystyle{cas-model2-names}

\begin{thebibliography}{99}

\bibitem{2019Recovering}
\href{https://doi.org/10.9734/ajpas/2019/v3i130081}{ Jeremy Becnel and Daniel Riser-Espinoza, Recovering a random variable from conditional expectations using reconstruction algorithms for the Gauss Radon transform, Asian J. Probab. Stat. 3 (1) (2019) 1–31.}

\bibitem{appl2017}
\href{https://doi.org/10.1016/j.apnum.2017.07.002}{ Xuenan Sun and Xuezhang Liang, Fast and exact 2d image reconstruction based on Hakopian interpolation, Appl. Numer. Math. 121 (2017) 185–197.}



\bibitem{NSTLAD211C0BE7E08794695A2123086A9B23}
\href{https://dx.doi.org/10.1088/0266-5611/16/3/307}{ Adel Faridani and Erik L. Ritman, High-resolution computed tomography from efficient sampling, Inverse Probl. 16 (3) (2000) 635–650.}
 
\bibitem{XU2006388}
\href{https://doi.org/10.1016/j.aam.2005.08.004}{Yuan Xu, A new approach to the reconstruction of images from Radon projections, Adv. Appl. Math. 36 (4) (2006) 388–420.}

\bibitem{WOS:000178752800002}
\href{https://doi.org/10.1007/s00041-002-0025-2}{Daniel Potts and Gabriele Steidl, Fourier reconstruction of functions from their nonstandard sampled Radon transform, J. Fourier Anal. Appl. 8 (6) (2002) 513–533.}


\bibitem{2023Random1}
\href{https://doi.org/10.1016/j.acha.2023.101576}{ Erika Porten, Juan M. Medina and Marcela Morvidone, Random sampling over locally compact Abelian groups and inversion of the Radon transform, Appl. Comput. Harmon. Anal. 67 (2023) 1–27.}


\bibitem{2025A}
\href{https://doi.org/10.1109/TGRS.2025.3552167}{Decheng Sun, Guijin Yao and Yue Li, DAS up- and downgoing wavefield separation via Radon transform combined with parallel u-network, IEEE Trans. Geosci. Remote Sensing. 63 (2025) 1-10.}




\bibitem{2024A}
\href{https://doi.org/10.1016/j.dsp.2023.104253}{Ming Long, Jun Yang, Saiqiang Xia , et al. A method for extracting micro-motion features of rotor targets based on GS-IRadon algorithm, Digit. Signal Prog. 144 (2024) 1-11.}




\bibitem{1983The}
\href{https://scholar.google.com/scholar?hl=zh-CN&as_sdt=0%2C14&q=The+Radon+transform+and+some+of+its+applications&btnG=}{ Stanley R. Deans, The Radon transform and some of its applications, Dover. 1983.}

\bibitem{2010A}
\href{https://doi.org/10.1007/s11856-010-0036-7}{Richard F. Bass and Karlheinz Gröchenig, Random sampling of bandlimited functions, Isr. J. Math. 177 (2010), 1--28.}




\bibitem{Entezari20121532}
\href{https://doi.org/10.1109/TMI.2012.2191417}{ Alireza Entezari, Masih Nilchian and Michael Unser, A box spline calculus for the discretization of computed tomography reconstruction problems, IEEE Trans. Med. Imaging. 31 (8) (2012) 1532--1541.}



\bibitem{appl2009}
\href{https://doi.org/10.1016/j.apnum.2008.03.038}{Ramakrishnan Radha and Suntharalingham Sivananthan, Local reconstruction of a function from a non-uniform sampled data, Appl. Numer. Math. 59 (2) (2009) 393--403.} 

\bibitem{2023Determination}
\href{https://doi.org/10.1016/j.jfa.2023.110151}{Youfa Li, Shengli Fan and Deguang Han, Determination of compactly supported functions in shift-invariant space by single-angle Radon samples, J. Funct. Anal. 285 (11) (2023) 1--38.} 

\bibitem{2013A}
\href{https://doi.org/10.1215/ijm/1403534485}{Richard F. Bass and Karlheinz Gröchenig, Relevant sampling of band-limited functions, Ill. J. Math. 57 (1) (2013) 43--58.}

\bibitem{2019A}
 \href{https://doi.org/10.1016/j.jat.2018.09.009}{Hartmut F\"uhr and Jun Xian, Relevant sampling in finitely generated shift-invariant spaces, J. Approx. Theory. 240 (2019) 1--15. }




\bibitem{2014Monte}
 \href{https://doi.org/10.1109/TIP.2014.2327813}{Stanley H. Chan, Todd Zickler and Yue M. Lu, Monte Carlo non-local means: Random sampling for large-scale image filtering, IEEE Trans. Image Process. 23 (8) (2014) 3711--3725.}

\bibitem{2001On}
\href{https://mathscinet.ams.org/mathscinet/article?mr=1864085}{Felipe Cucker and Stephen J. Smale, On the mathematical foundations of learning, Bull. Amer. Math. Soc.  39 (1) (2002) 1--49.}

\bibitem{2019Reconstruction}
\href{https://doi.org/10.1088/1361-6420/ab40f7}{Yaxu Li, Jinming Wen and Jun Xian, Reconstruction from convolution random sampling in local shift invariant spaces, Inverse Probl. 35 (12) (2019) 1--15.} 


\bibitem{2023Random}
\href{https://doi.org/10.1007/s00025-023-01865-y}{Yingchun Jiang and Haiying Zhang, Random sampling in multi-window quasi shift-invariant spaces, Results Math. 78 (2023) 1--20.}

\bibitem{2023AA}
\href{https://doi.org/10.1090/proc/16330}{Yaxu Li, Random sampling in reproducing kernel spaces with mixed norm, Proc. Amer. Math. Soc. 151 (6) (2023) 2631--2639.}

\bibitem{2013AA}
 \href{https://doi.org/10.1090/S0002-9939-2012-11644-2}{Zuhair M. Nashed, Qiyu Sun and Jun Xian, Convolution sampling and reconstruction of signals in a reproducing kernel subspace, Proc. Amer. Math. Soc. 141 (6) (2013) 1995--2007. }




\bibitem{2021A}
\href{https://doi.org/10.1016/j.acha.2021.03.006}{Yaxu Li, Qiyu Sun and Jun Xian, Random sampling and reconstruction of concentrated signals in a reproducing kernel space, Appl. Comput. Harmon. Anal.  54 (2021) 273--302.}



\bibitem{1981Hough}
\href{https://doi.org/10.1109/TPAMI.1981.4767076}{Stanley R. Deans, Hough transform from the Radon transform, IEEE Trans. Pattern Anal. Mach. Intell. 3 (2) (1981) 185--188.}




\bibitem{WOS:000306433100001}
\href{https://doi.org/10.1007/s10208-011-9099-z}{ Joel A. Tropp, User-friendly tail bounds for sums of random matrices, IEEE Trans. Inf. Theory. 12 (2012) 389–434.}









\end{thebibliography}


\end{document}